\documentclass{elsarticle}%
\usepackage{amsfonts}
\usepackage{amsthm}
\usepackage{amsmath}
\usepackage{amssymb}
\usepackage{bbold}
\usepackage{mathtools}
\usepackage{xcolor}
\usepackage{soul}
\usepackage{mathrsfs}
\usepackage{tabularx}
\usepackage[top=1in, bottom=1in, left=1in, right=1in]{geometry}
\usepackage{graphicx}%
\providecommand{\U}[1]{\protect\rule{.1in}{.1in}}

\newtheorem{theorem}{Theorem}[section]
\newtheorem{lemma}{Lemma}[section]
\newtheorem{corollary}{Corollary}[section]

\theoremstyle{definition}
\newtheorem{definition}{Definition}[section]
\theoremstyle{remark}

\numberwithin{equation}{section}
\begin{document}
	\begin{frontmatter}
		
		\title{ A Pillai-Catalan-type problem involving Fibonacci numbers.
		}
		
		
		
		\author[]{Seyran S. Ibrahimov}
		\ead{Seyran.Ibrahimov@emu.edu.tr}
		\author[]{Nazim I. Mahmudov}
		\ead{nazim.mahmudov@emu.edu.tr}
		
		\address{Department of Mathematics, Eastern Mediterranean University, Mersin 10, 99628, T.R. North Cyprus, Turkey}
		\address{Research Center of Econophysics, Azerbaijan State University of Economics (UNEC), Istiqlaliyyat Str. 6, Baku 1001,Azerbaijan}
		
		

\begin{abstract}	
			This paper addresses  A Pillai-Catalan-type problem assosiated with Fibonacci numbers. Let $F_{n}$ be the Fibonacci numbers defined by the recurrence relation $F_{1}=1$, $F_{2}=1$ and $F_{n}=F_{n-1}+F_{n-2}$ for all $n
		\geq 3$. We will  find all positive integer solution to the equation $3^{x}-F_{n}2^{y}=1$ using properties of Fibonacci numbers, linear forms of logarithms and Baker-Davenport reduction method. 
		\end{abstract}

	\end{frontmatter}
	
	\section{INTRODUCTION}
Diophantine equation theory is one of the important components of number theory with a long history that dates back to ancient times. The title of Diophantine is associated with the name Diophantus, who lived in the third century and conducted the earliest studies in this field. In general, all integer solutions to these equations are required to be found. However, Hilbert conjectured in 1900 that 

"there is no general method for solving Diophantine equations in the set of integers using finite numerical operations". 

This is the tenth of Hilbert's 23 famous problems, which had been solved in 1970 by Russian mathematician Y. Matiyasevich. Therefore, researchers construct different methods to solve specific classes of Diophantine equations. In addition, it's no wonder that several long-standing and very interesting problems related to this subject have yet to be fully resolved. One of this problems is the following equation offered by Pillai in 1936 (see, e.g.\cite{1}-\cite{2}).
	\begin{equation}\label{de}
		 m^{x}-n^{y}=c
	\end{equation}

	Pillai conjectured that equation (\ref{de}) has for fixed integer $c\geq1$ at most finitely many positive integer solutions $(m,n,x,y)$, with $x,y\geq 2$. Mihailescu provided a solution to this equation in the case where $c=1$ , famously referred to as the Catalan conjecture \cite{3}. In particular, before Pillai proposed the conjecture, Herschfeld  \cite{8},\cite{20} had analyzed a special case with $(m,n)=(2,3)$. Furthermore, it was proven by Stroeker and Tijdeman \cite{22} that equation (\ref{de}) has at most one solution for $m=2$ and $n=3$ when $c>13$.
	
	To prove our primary result, we will use lower bounds for linear forms in logarithms (Baker's theory) and a variation of the Baker-Davenport reduction method in diophantine approximation which are important tools for solving diophantine equations. Several authors, particularly in the recent decade, have used these tools to investigate several Diophantine equations involving recurrence number sequences (see,\cite{9}-\cite{16},\cite{18} eg). 
	
	In this paper, we study the Diophantine equation
	\begin{equation}\label{pi}
		3^{x}-F_{n}2^{y}=1
	\end{equation}
\begin{theorem} The only positive integer solutions to the equation (\ref{pi}) are
	\begin{align*}
		(x,y,n)=\lbrace(1,1,1),(1,1,2),(2,2,3),(2,3,1), (2,3,2),(3,1,7), (4,4,5)\rbrace.
	\end{align*}

\end{theorem}

 The version of equation (\ref{pi}) for Lucas numbers has been studied in (\cite{19}). Also, following particular cases of equation (\ref{pi}) have previously been examined in (\cite{15}).
 
 	$\bullet$ $y=1$ in equation (\ref{pi}) yields the following equation:
	\begin{equation}		
	F_{n}=\frac{3^{x}-1}{2},
\end{equation}
	$\bullet$ $y=2$ and $x=2k$ in equation (\ref{pi}) yields the following equation: 
	\begin{equation}		
	F_{n}=\frac{9^{k}-1}{4},
\end{equation}
	$\bullet$ $y=3$ and $x=2k$ in equation (\ref{pi}) yields the following equation: 

	\begin{equation}		
	F_{n}=\frac{9^{k}-1}{8}.
\end{equation}

 Furthermore, it is demonstrated that in (\cite{15}) $(x,n)=\lbrace (1,1),(1,2),(3,7)\rbrace$ are only solutions of equation (1.3), $(k,n)=(1,3) $ is the unique solution of equation (1.4), and $(k,n)=\lbrace (1,2),(1,1)\rbrace$ are the only solutions of equation (1.5). Thus, it follows that
 $$(x,y,n)=\lbrace(1,1,1),(1,1,2),(2,2,3),(2,3,1), (2,3,2),(3,1,7)\rbrace$$
 
  are only solutions to the equation (\ref{pi}) for the cases $y=1,2,3.$

	\section{PRELIMINARIES}	
	This section presents some fundamental concepts that will be required in the proof of our main result.
	
	\begin{lemma} Let $t, u, l ,k$ be some positive integers then
		
		1) $3^{2t-1}-1=2A_{t}$	
		
		2) $3^{2t-1}+1=4B_{t}$
		
		3) $3^{2t}+1=2C_{t}$	
		
		4) $3^{8t-2}-1=8D_{t}$
		
		5) $3^{8t-4}-1=16F_{t}$
		
		6) $3^{8t-6}-1=8R_{t}$
		
		7) $3^{8t}-1=32S_{t}$, for all $t=2k-1$
		
		8) $3^{8t}-1=2^{l+4}Z_{t}$, for all $t=2^{l-1} u$

		Here, the odd numbers $A_{t}$, $B_{t}$, $C_{t}$, $D_{t}$, $F_{t}$,	 $R_{t}$, $S_{t}$, $K_{t}$ and positive integer $Z_{t}$ are depends the value of  $t$. 
		
		\begin{proof} We will now present the proof of the lemma's parts 1, 3, and 8. Analogously, other sections of the lemma can be demonstrated.
			\bigskip
						
			$	1) \quad  3^{2t-1}-1=2 \sum_{k=2}^{2t}3^{2t-k}=2A_{t},$
			
			\bigskip
			
			$	3) \quad  3^{2t}+1=3(2A_{t}+1)+1=2(3A_{t}+2)=2C_{t},$
			
		\bigskip
		
			If we use part 3), we obtain
				
				\bigskip
				
			$	8) \quad 3^{2^{l+2}}-1=2\prod_{k=0}^{l+1}\big(3^{2^{k}}+1\big)=2^{l+4}K_{t},$
			
		\bigskip
		
			    and
			   	
			   \bigskip
			   
			$ 3^{u\cdot 2^{l+2}}-1=(3^{2^{l+2}})^{u}-1=2^{l+4}Z_{t}.$
		\end{proof}
\end{lemma}
As the aforementioned references demonstrate, solving the kind of diophantine equation we studied in this article typically requires many computations. However, in this work, it is simpler for us to ascertain effective bounds for the variables in equation (1.2) thanks to the use of Lemma 2.1 and its following result: 

\begin{corollary}(\cite{4,21})
	 If $3^{t}-1\equiv0\quad(\bmod2^{l})$ and $l\geq 3$ then $t_{min}=2^{l-2}$ and $t=2^{l-2}u, \quad u\geq1$.
\end{corollary}
\begin{proof}
	
	By Euler's theorem we have:
	\begin{align*}
		3^{2^{l-1}}-1\equiv0\quad(\bmod2^{l})
	\end{align*}
Furthermore, we are aware that if $d$ is the smallest positive integer such that

\begin{align*}
	3^{d}-1\equiv0\quad(\bmod2^{l})
\end{align*}
then
\begin{align}
	d|2^{l-1}
\end{align}
 
 Moreover, by applying property (2.1) and part 3 of Lemma 2.1, we determine that  $t_{min}=2^{l-2}$.

\end{proof}
\begin{lemma}\textbf{(Binet's formula)}(\cite{5})
	 The nth term of the Fibonacci sequence can be determine with the following formula:
	
	\begin{align*}		
		F_{n}=\frac{(\alpha)^{n}-(-\alpha)^{-n} }{\sqrt5}, \quad n\geq1
	\end{align*}
where
\begin{align*}
	 \alpha=\frac{1+\sqrt5}{2}.
\end{align*}
\end{lemma}
Using the Binet's formula, we can straightforwardly derive the following inequalities

\begin{corollary}(\cite{5})
	\begin{align*}
		\alpha^{n-2}\leq F_{n} \leq \alpha^{n-1},\quad n\geq1.
		\end{align*}
\end{corollary}
  We shall now provide a few basic concepts from algebraic number theory.

Let $z$ be an algebraic number of degree d with minimal polynomial

	\begin{align*}
		a_{0}x^{d}+a_{1}x^{d-1}+...+a_{d}=a_{0}\prod_{i=1}^{d}(X-z^{(i)})
	\end{align*}
where the $a_{1}$, $a_{2}$,..., $a_{d}$ are relatively prime integers with $a_{0}> 0$ and $z^{(1)}$, $z^{(2)}$,...,  $z^{(d)}$  are conjugates of $z$. 

\begin{definition}
	The logarithmic height of $z$ is defined by
	\begin{align*}
		h(z)=\frac{1}{d}\bigg(\log a_{0}+\sum_{i=1}^{d} \log\big(\max\lbrace \vert z^{(i)}\vert,1\rbrace\big)\bigg)
	\end{align*}
\end{definition}
Additionally, we present the Matveev`s theorem, which gives a large upper bound for the variable $n$ in the equation (\ref{pi}).
\begin{theorem}(\textbf{Matveev`s Theorem})(\cite{6})
	Assume that $\gamma_{1},\dots, \gamma_{n}$ are positive algebraic numbers in a real algebraic number field $\mathbb{L}$ of degree $d_{\mathbb{L}}, r_{1},\dots, r_{n}$ are rational integers, and 
	\begin{align*}
		\Lambda := \gamma_{1}^{r_{1}}\dots \gamma_{n}^{r_{n}}-1,
	\end{align*} 
is not zero. Hence 
\begin{align}
	\vert \Lambda\vert > \exp \bigg(-1.4 \cdot 30^{n+3} \cdot n^{4.5} \cdot d_{\mathbb{L}}^{2}(1+\log d_{\mathbb{L}})(1+\log T)A_{1}\dots A_{n}\bigg),
\end{align}
where $T\geq \max\lbrace \vert r_{1}\vert ,\dots, \vert r_{n}\vert\rbrace$, and $A_{j}\geq \max\lbrace d_{\mathbb{L}}h(\gamma_{j}), \vert \log \gamma_{j}\vert, 0.16\rbrace$, for all $j=1,\dots,n$.
\end{theorem}

Furthermore, we present the following valuable lemma, which was demonstrated by Dujella and Pethő (\cite{7}) and is a variation of a lemma of Baker and Davenport (\cite{17}). We use it to reduce the upper bound for the variable $n$ in the equation (\ref{pi}). For real number $y$, we put $\Vert y \Vert$=$\min\lbrace\vert y-n\vert: n\in Z\rbrace$, the distance from $y$ to the nearest integer.
\begin{lemma}
	Let $M$ be a positive integer, $\frac{p}{q}$ be a convergent of the continued fraction of the irrational $\gamma$ such that $q>6M$, and let $A,B,\mu$ be some real numbers with $A>0$ and $B>0$. If $\varepsilon=\Vert \mu q \Vert-M\Vert \gamma q\Vert>0$, hence there is no solution to the inequality 
	\begin{align*}
		0<\vert u\gamma-v+\mu\vert<AB^{-\omega},
	\end{align*}
in positive integers $u,v$ and $\omega$ with 
\begin{align*}
	u\leq M\quad and \quad \omega\geq \frac{\log(\frac{Aq}{\varepsilon})}{\log B}.
\end{align*}
\end{lemma}
\section{ Proof of Theorem 1.1.}

	First, we examine the equations 
	$F_{n}=\frac{3^{2k-1}-1}{4}$ and $F_{n}=\frac{3^{2k-1}-1}{8}$. We can quickly determine that there is no solution to either equation by applying Lemma 2.1 (part 1). Also, Herschfeld (see, e.g \cite{8},\cite{20}) has already solved the equation $3^{x}-2^{y}=1$ for cases $n=1, 2$ and the only solutions to this equation are $(x,y)=\lbrace(1,1),(2,3)\rbrace.$
	
	Then, we assume that $y\geq 4$ and $n\geq 3$. By using Corollary 2.1, we obtain $x\geq2^{y-2}\geq y$. Then we get, $n\geq y$. By utilizing Corollary 2.2, we have
	\begin{align*}
			x<n\frac{\log \alpha}{\log 3}+y\frac{\log 2}{\log 3}<n\bigg(\frac{\log \alpha+\log 2}{\log 3}\bigg),  
		\end{align*}		
		which implies
		\begin{equation}
		 x<1.1n.
	\end{equation}

We can now express equation (\ref{pi}) as follows by using Binet's formula:
\begin{align*}
	&3^{x}-F_{n}2^{y}=3^{x}-\frac{\alpha^{n}-(-\alpha)^{-n}}{\sqrt{5}} 2^{y}=1,\\
\end{align*}
	After taking the modules on both sides of the last equality and performing the required computations, we acquire
	\begin{align*}
	&\bigg\vert 3^{x}-\frac{\alpha^{n}}{\sqrt{5}}2^{y}\bigg\vert \leq 1+2^{y-1}.\\
\end{align*}
	Dividing both sides of the above inequality by $\frac{\alpha^{n}}{\sqrt{5}}2^{y}$ we get
	\begin{equation}
	\Lambda :=\vert 3^{x}\sqrt 5 \alpha^{-n}2^{-y}-1\vert< \frac{\sqrt{5}}{\alpha^{n}}.
\end{equation}
Next, let us apply Matveev's theorem  to the left side of the above inequality with $\gamma_{1}=3, \gamma_{2}=\sqrt{5},\gamma_{3}=\alpha, \gamma_{4}=2$ and $r_{1}=x, r_{2}=1, r_{3}=-n, r_{4}=-y$. Since $\gamma_{1}, \gamma_{2}, \gamma_{3}$ and $\gamma_{4}$ belongs to real quadratic number field $\mathbb{L}$=$\mathbb{Q}(\sqrt{5})$ we take $d_{\mathbb{L}}=2$. It is obvious that $\Lambda$ is nonzero. If $\Lambda=0$ then we have that $\frac{3^{x}}{2^{y}}=\frac{\alpha^{n}}{\sqrt{5}}$, which indicates that the right hand side, $\frac{\alpha^{n}}{\sqrt{5}}\in\mathbb{Q}$, however this is a contradiction. Since $h(\gamma_{1})=\log{3}, h(\gamma_{2})=\frac{1}{2}\log{5}, h(\gamma_{3})=\frac{1}{2}\log{\alpha}, h(\gamma_{4})=\log{2}$ we take $A_{1}=2\log{3}, A_{2}=\log{5}, A_{3}=\log{\alpha}, A_{4}=2\log{2}$. In addition, since $x<1.1n$, we take $T=1.1n$. 

Moreover, by combining (2.2) and (3.2), we have the inequality that follows

\begin{align*}
	&\frac{\sqrt{5}}{\alpha^{n}}>\exp \bigg(-1.4\cdot 30^{7} \cdot 2^{13} (1+\log{2})(1+\log{(1.1n)})\log{2}\log{3}\log{5}\log{\alpha})\bigg),\\
\end{align*}
A brief calculation reveals that
\begin{align*}
	&n<1.5\cdot 30^{7} \cdot 2^{13} (1+\log{2})(1+\log{(1.1n)})\log{2}\log{3}\log{5},\\
\end{align*}
Thus we have that
\begin{equation}
	n<2.16\cdot 10^{16}.
\end{equation}
Corollary 2.1, in combination with inequalities (3.1) and (3.3), states that
\begin{align*}
 &2^{y-2}<1.1\cdot 2.16\cdot 10^{16}.\\
\end{align*}
  Therefore, we conclude that $y<56$.

   Now, we reduce the upper bound of $n$ by applying Lemma 2.3.
   Let
\begin{equation}
	\Gamma :=n\log{\alpha}-x\log{3}+\bigg(y\log{2}-\log{\sqrt{5}}\bigg)
\end{equation}
So that
 \begin{align*}
 	\Lambda :=	\vert\exp(\Gamma) -1\vert< \frac{\sqrt{5}}{\alpha^{n}}<\frac{1}{2}, \quad for\quad\forall n\geq 3.
 \end{align*}
and we have
\begin{align*}
	2\vert z\vert>\vert\log(1+z)\vert ,\quad \forall z\in \bigg(-\frac{1}{2},\frac{1}{2}\bigg)
\end{align*}
Then, applying the last inequality with $z=exp(\Gamma) -1$, we get
\begin{align}
	\vert\Gamma \vert<\frac{2\cdot\sqrt{5}}{\alpha^{n}}
\end{align}
  
 Therefore, from (3.4) and (3.5),
\begin{align*}
	\bigg\vert n\frac{\log{\alpha}}{\log{3}}-x+\bigg(y\frac{\log{2}}{\log{3}}-\frac{\log{\sqrt{5}}}{\log{3}}\bigg)\bigg\vert<\frac{5}{\alpha^{n}}.
\end{align*}

Now, we take $M=2,16\cdot10^{16}$ and $\gamma=\frac{\log{\alpha}}{\log{3}}$.
\bigskip

Then in the continued fraction expansion of $\gamma$, we take $q_{41}$, the denominator of the $41th$ convergent of $\gamma$, which exceeds $6M$. Now, we use Mathematica program to calculate
\begin{align*}
	\varepsilon_{y}:=\bigg\Vert\bigg(y\frac{\log{2}}{\log{3}}-\frac{\log{\sqrt{5}}}{\log{3}}\bigg)q_{41}\bigg\Vert-2.16\cdot 10^{16}\bigg\Vert \frac{\log{\alpha}}{\log{3}}q_{41}\bigg\Vert,
\end{align*}

for each $y\in \lbrace{4,5,...,56}\rbrace$ where $q_{41}=1116972345258589541$ and we get that

$$\varepsilon_{y}>0.0296-0.02=0.0096,\quad y\in \lbrace{4,5,...,56}\rbrace.$$

Let 
 $ A:=5,\quad B=:\alpha>1\quad and \quad \omega=:n $. Then from Lemma 2.3, we find that
 
 \begin{align*}
 	n<\frac{\log{\bigg(\frac{5\cdot 1116972345258589541}{0.0096}\bigg)}}{\log{\alpha}}<94.
 \end{align*}

Hence, $x<1.1n<103$. We determine that possible solutions to equation (\ref{pi}) satisfy the following inequalities: 
\begin{equation}
	\begin{cases}
		3<y<56,\\
		x<103,\\
		2<n<94.
	\end{cases}
\end{equation}
A run using Mathematica program revealed that the only positive integer solution to equation (\ref{pi}) with the conditions (3.6) is
\begin{align*}
	(x,y,n)=(4,4,5).
\end{align*}
The proof is completed.

\end{document}